\documentclass[11pt]{amsart}
\usepackage{amsfonts,amsmath}
\usepackage[usenames]{color}
\usepackage[dvips]{graphicx}
\newdimen\margin   
\def\COMMENT#1{}
\let\COMMENT=\footnote

\newcommand{\eps}{\varepsilon}

\newcommand{\prob}{\mathbb{P}}
\newcommand{\ex}{\mathbb{E}}

\newcommand{\U}{{\mathcal U}}
\newcommand{\B}{{\mathcal B}}

\newcommand{\E}{{\mathcal E}}

\newcommand{\F}{{\mathcal F}}
\newcommand{\M}{{\mathcal M}}
\newcommand{\C}{{\mathcal C}}
\newcommand{\D}{{\mathcal D}}

\newcommand{\A}{{\mathcal A}}

\newcommand{\Xf}{{\bf X}}
\newcommand{\G}{\mathbb{G}(n,d)}
\newcommand{\MG}{\widetilde{\mathbb{G}}(n,d)}

\newcommand{\N}{{\mathcal N}}
\renewcommand{\P}{{\mathcal P}}
\newcommand{\I}{{\mathcal I}}

\newcommand{\lam}{\lambda}
\newcommand{\de}{\mathbf{d}_S}

\numberwithin{equation}{section}

\newtheorem{firsttheorem}{Proposition}

\newtheorem{theorem}[firsttheorem]{Theorem}
\newtheorem{lemma}[firsttheorem]{Lemma}
\newtheorem{corollary}[firsttheorem]{Corollary}

\newtheorem{proposition}[firsttheorem]{Proposition}

\begin{document}
\title{Rumor Spreading on Random Regular Graphs and Expanders}
\author{Nikolaos Fountoulakis \and Konstantinos Panagiotou}

\begin{abstract}
Broadcasting algorithms are important building blocks of distributed systems. In this work we investigate the typical performance of the classical and well-studied \emph{push model}. Assume that initially one node in a given network holds some piece of information. In each round, every one of the informed nodes chooses independently a neighbor uniformly at random and transmits the message to it.


In this paper we consider random networks where each vertex has degree $d \ge 3$, i.e., the underlying graph is drawn uniformly at random from the set of all $d$-regular graphs with $n$ vertices. We show that with probability $1 - o(1)$ the push model broadcasts the message to all nodes within $(1 + o(1))C_d\ln n$ rounds, where $$C_d = \frac1{\ln(2(1-\frac1d))} - \frac1{d\ln(1 - \frac1d)}.$$
Particularly, we can characterize precisely the effect of the node degree to the typical broadcast time of the push model. Moreover, we consider pseudo-random regular networks, where we assume that the degree of each node is very large. There we show that the broadcast time is $(1+o(1))C\ln n$ with probability $1 - o(1)$, where $C = \lim_{d\to\infty}C_d = \frac1{\ln2} + 1$.
\end{abstract}

\maketitle

\section{Introduction}

\subsection{Rumor Spreading and the Push Model}

In this work we consider the classical and well-studied \emph{push model} (or \emph{push protocol}) for disseminating information in networks. Initially, one of the nodes obtains some piece of information. In each succeeding round, every node who has the information passes it another node, which it chooses independently and uniformly at random among its neighbors. The important question is: how many rounds are typically needed until all nodes are informed?

The push model has been the topic of many theoretical works, and its performance was evaluated on several types of networks. In the case where the underlying network is the complete graph, Frieze and Grimmett~\cite{FrGrim85} proved that with high probability (whp.) (i.e., with probability $1-o(1)$) 
the broadcasting is completed within $(1 + o(1))(\log_2 n + \ln n)$ rounds, where $n$ denotes the total number of nodes. Recently, this result was extended by the two authors and Huber~\cite{FHP09} to the classical Erd{\H{o}}s-R{\'e}nyi graph $G_{n,p}$, which is obtained by including each of the possible $\binom{n}{2}$ edges with probability $p$, independently of all other edges. Among other results, they showed that if $p = \omega(\frac{\ln n}{n})$, then the typical broadcast time essentially coincides with the broadcast time on the complete graph. In other words, as long as the average degree of the underlying graph is significantly larger than $\ln n$, the number of rounds needed is not affected. However, prior to this work, there was no result describing the performance of the push model on significantly sparser networks.

The typical broadcast time of the push model was also investigated for other types of networks, albeit not as precisely. Feige et al.\ derived in~\cite{Feig91} bounds that hold for arbitrary graphs. Moreover, they proved a logarithmic upper bound for the number of rounds needed to broadcast the information if the underlying network is a hypercube. This result was generalized by Els\"asser and Sauerwald, who determined in \cite{ES07} similar bounds for several classes of Cayley graphs. 
Bradonjic et\ al.~\cite{BEFSS10} considered random geometric graphs as underlying networks, and proved that whp. the broadcast time is essentially proportional to the diameter of these graphs.

\subsection{Our Contribution}

The main contribution of this paper is the precise analysis of the push model on sparse random networks. Note that in this context the study of the $G_{n,p}$ distribution is not appropriate, as we would have to set $p= c/n$ for some constant $c>0$. However, for such $p$ the random graph $G_{n,p}$ is typically not connected. In fact, if we took any $p=o\left({\ln n \over n}\right)$, we would face the same problem, as such a $p$ is below the connectivity threshold for $G_{n,p}$ (see for example~\cite{JLR}). 
 
A candidate class of random graphs that combines the feature of constant average degree with that of connectivity is the class of random $d$-regular graphs $\G$ for $d\geq 3$. It is well-known that a random $d$-regular graph on $n$ vertices
is connected with probability $1-o(1)$. Thus, a typical member of this class of graphs is suitable for the analysis of the push protocol as far as the effect of density is concerned. 
Let $T = T(\G)$ denote the broadcast time of the push model on $\G$. Note that in this case the choice of the 
vertex where the information is placed initially does not matter. 
\begin{theorem} \label{thm:RandReg}
With probability $1-o(1)$ 
$$ |T(\G) - C_d\ln n| = O((\ln\ln n)^2), $$  
where $C_d = \frac1{\ln(2(1-\frac1d))} - \frac1{d\ln(1 - \frac1d)}$.
\end{theorem}
The above theorem is interpreted as follows: for almost all $d$-regular graphs on $n$ vertices, with probability $1-o(1)$ the 
push protocol broadcasts the information within the claimed number of rounds. 
It is easy to see that as $d$ grows $C_d$ converges to ${1\over \ln 2} + 1$, which is the constant factor of the
broadcast time of the push protocol on the complete graph, as shown by Frieze and Grimmett~\cite{FrGrim85}. 
Thus our result reveals the essential 
insensitivity of the performance of the push protocol regarding the density of the underlying network and shows that 
the crucial factor is the ``uniformity" of its structure. 

We explore further this aspect and we consider regular graphs whose structural characteristics resemble those of a regular random graph. 
In particular, we consider expanding graphs whose ``geometry" is determined by the spectrum of their adjacency matrix. 

In Subsection~\ref{sec:explanation} below we give an intuitive description of the evolution of the randomized protocol, thus 
explaining also how do the two summands involved in $C_d$ come up. 

\subsubsection{Regular expanding graphs}

Expanding graphs have found numerous applications in modern theoretical computer science as well as in pure mathematics. 
Their properties together with the theory of finite Markov chains have led to the solution of central problems such as 
the approximation of the volume of a convex body, approximate counting or 
the approximate uniform sampling from a class of combinatorial objects. The latter applications have had further impact outside 
computer science such as in the field of statistical physics.
We refer the reader to the excellent survey of Hoory et al.~\cite{Hoor06} for a detailed exploration of the properties and the numerous applications 
of expanding graphs.

The main feature of an expanding graph is that every set of vertices is connected to the rest of the graph by a large number of 
edges. This key property makes random walks on such graphs rapidly mixing and has led to the above mentioned applications. 
Moreover, this property makes expanding graphs an attractive candidate for communication networks.   
Intuitively, the high expansion of a graph implies that information that is initially located on a small part of the
graph can be spread quickly on the rest of the graph. This becomes possible as the high expansion of a graph
ensures the lack of ``bottlenecks", that is, local obstructions on which a broadcast protocol would need a 
significant amount of time in order to bypass them. 

We focus on a spectral characterization of expanding graphs, which is related to the spectral gap of their adjacency 
matrix. 
Let $G=(V,E)$ be a connected $d$-regular graph and let $A$ be its adjacency matrix. 
The Perron-Frobenius Theorem implies (see Proposition~2.10 in~\cite{KrivSud}) that the largest eigenvalue of $A$ equals $d$ and that the corresponding eigenvector is proportional to the all-ones vector $[1,\ldots, 1]^T$. Let 
$\lam_1,\ldots, \lam_n$ be the eigenvalues of $A$ ordered according to their value (note that since $A$ is symmetric, these are 
all real). Set
$$ \lam:= \lam(A) := \max_{2\leq i \leq n} |\lam_i|. $$  
If $G$ has $n$ vertices we say that $G$ is an $(n,d,\lam)$ \emph{graph}. 
One can show (see for example~p.\ 19 in~\cite{KrivSud}) that $\lam = \Omega (\sqrt{d})$. 
In particular, Alon and Boppana, Nilli~\cite{Nil93} and Friedman~\cite{Fried93} have shown that for every 
$d$-regular graph on $n$ vertices we have $\lam_2 \geq 2 \sqrt{d-1}(1-o(1))$. 
 
We are interested in the class of $d$-regular graphs for which $\lam$ almost attains this lower bound. 
In particular, we are concerned with the broadcast time of the randomized protocol on expanding $d$-regular graphs on 
$n$ vertices with $\lam = O(\sqrt{d})$. Such graphs can be explicitly constructed through number-theoretic or group theoretic methods 
(see the survey of Krivelevich and Sudakov~\cite{KrivSud} where numerous examples are presented).
Informally, we show that if $d = \omega (\sqrt{n})$, then the broadcast time is essentially the broadcast 
time on the complete graph with $n$ vertices.
\begin{theorem} \label{thm:pseudo}
Let $G$ be a connected $(n,d,\lam)$ graph with $\lam \leq C \sqrt{d}$ and $d \geq {2 C}\sqrt{n \ln^{1/9}n}$.
Then for any $v \in V$, with probability $1-o(1)$
$$|T(G,v) - (\log_2n + \ln n) | = o(\ln n).$$  
\end{theorem}
Again, this theorem shows the insensitivity of the broadcast time on the density of the underlying network. In fact, the assumption that 
$\lam = O(\sqrt{d})$ does not merely yield the high expansion of the graph, but it also implies that the edges of 
the graph are distributed in a uniform way among each subset of vertices. As we shall see in the proof of Theorem~\ref{thm:pseudo},
this assumption implies that the structure of the graph is not very different from that of a random graph on $n$ vertices 
and edge probability equal to $d/n$. For example, the number of edges between a subset $S$ and  its complement is close 
to ${d\over n}|S|(n-|S|)$, which is the expected value in the random graph with edge probability $d/n$. 
In this sense, such graphs are \emph{pseudorandom}. This notion was introduced by Thomason~\cite{Thom85} and was explored further by Chung, Graham and Wilson~\cite{CGW89}, especially regarding its spectral characterization.

\subsection{The evolution of the randomized protocol in a nutshell} \label{sec:explanation}


Roughly speaking, the evolution of the protocol consists of three phases, which have different characteristics regarding the rate in which the information is spread.

Let us consider the first phase, which ends when there are at least $\eps n$ informed vertices, for some very small $\eps > 0$. Let us denote by $\I_t$ the set of informed vertices (i.e., those who possess the information), and by $\U_t$ the set of uninformed vertices at the beginning of round $t+1$ of the push model. Moreover, let $e$ be some edge that is incident to a vertex in $\I_t$ that \emph{has not been used up to now} to transmit a message, and let $\E_t$ be the set of such edges. Then we show that the subgraph of $\G$ induced by $\I_t$ is essentially a \emph{tree}, and moreover, that $\E_t$ contains $\approx 2^t(1 - \frac1d)^t$ edges. To see this, note that as every vertex informs some specific neighbor with probability $1/d$, the expected number of edges from $\E_t$ that are going to be used is $|\E_t|/d$. This means that $\approx |\E_t|/d$ new vertices are going to be informed (as the set of informed vertices induces a tree), implying that $|\E_{t+1}| \approx |\E_t| - |\E_t|/d + (d-1)|\E_t|/d$, as for every vertex that becomes informed in this round the number of edges counted in $\E_t$ increases by $d-1$. So, $|\E_{t+1}| \approx 2(1 - \frac1d)|\E_t|$. Note that in this calculation we worked only with expected values. In the actual proof we have to show that all the relevant quantities are sharply concentrated around their expectations. To this end, we use a variant of \emph{Talagrand's} inequality by McDiarmid~\cite{McD02} (Theorem~\ref{thm:McDiarmid}), which has not been used very frequently in the analysis of distributed algorithms. We believe that it could be widely applicable to the analysis of existing or future randomized
protocols with several different degrees of dependency.

As soon as the number of informed vertices is $\ge \eps n$, then after very few rounds the number of informed vertices is already 
$(1 - \eps)n$. Here it is essentially the expansion properties of $\G$, which guarantee that every large set of vertices has 
linearly many neighbors and, thus, with high probability a certain fraction of those become informed in each round.

During the final phase, the number of remaining uninformed vertices shrinks by a factor of $(1 - \frac1d)^d$. Indeed, suppose that
 there are $o(n)$ uninformed vertices. Then we expect that almost all of them have the property that the number of their neighbors in $\I_t$ is
 $d$, implying that the probability that any one of the remains uninformed is precisely $(1 - \frac1d)^d$. An easy calculation shows that a
 ``typical'' subset of $\G$ has this property. However, the set of uninformed vertices might not be typical at all, implying that we need
 additional effort to guarantee the desired properties.

\section{Concentration inequalities}

In this section we will state two concentration inequalities that will serve as the backbone of our proofs. 
The first one is a Chernoff-type bound for sums of negatively correlated random variables, see e.g.~\cite{book:dp09}. 
\begin{theorem}\label{thm:Chernoff}
Let $I_1,\ldots, I_n$ be a family of indicator random variables on a common probability space, 
which are identically distributed and negatively correlated, i.e., $\ex(I_i I_j) \le \ex(I_i)\ex(I_j)$ for all $1 \le i,j \le n$. 
Let $X := \sum_{i=1}^n I_i$. Then, for any $t>0$
\begin{equation*}
\prob \left( |X - \ex (X) | > t \right) < 2 \exp \left(-{t^2 \over 2\left( \ex(X) + t/3\right)} \right).
\end{equation*}
\end{theorem}
The next concentration inequality that we will need is due to McDiarmid~\cite{McD02}, and it is based on the work of 
Talagrand~\cite{Tal95}. We give first a few necessary definitions. Let $B$ be a finite set and let $Sym(B)$ be the set of all permutations on $B$. Assume that $\pi$ is an element of $Sym(B)$, drawn uniformly at random. 
Also, let $\Xf = (X_1,\ldots,X_n)$ be a finite family of independent random variables, where $X_j$ takes values in a set $\Omega_j$. Finally, set $\Omega = Sym(B) \times \prod_{j=1}^n \Omega_j$.
\begin{theorem}\label{thm:McDiarmid} 
Let $c$ and $r$ be positive constants. Suppose that $h:\Omega \to \mathbb{R}_+$ satisfies the following conditions. 
For each $(\sigma, {\bf x}) \in \Omega$ we have 
\begin{itemize}
\item if ${\bf x}'$ differs from $\bf x$ in only one coordinate, then $|h(\sigma,{\bf x}) - h(\sigma,{\bf x}')| \le 2c$;
\item if $\sigma'$ can be obtained from $\sigma$ by swapping two elements, then $|h(\sigma, {\bf x}) - h(\sigma',{\bf x})| \le c$;
\item if $h(\sigma, {\bf x}) = s$, then there is a set of at most $rs$ coordinates such that $h(\sigma', {\bf x}')\geq s$ 
for any $(\sigma', {\bf x}') \in \Omega$ that agrees with $(\sigma, {\bf x})$ on these coordinates.  
\end{itemize}
Let $Z = h(\pi, \Xf)$ and let $m$ be the median of $Z$. Then, for any $t>0$ 
\begin{equation*}
\prob \left( |Z-m|>t \right) \leq 4 \exp \left(- {t^2 \over 16rc^2(m+t)} \right). 
\end{equation*}
\end{theorem}

\section{Properties of random regular graphs and the configuration model}

\subsection{The configuration model}
We perform the analysis of the randomized protocol using the \emph{configuration model} introduced by
Bender and Canfield~\cite{BenCan} and independently by Bollob\'as~\cite{Bol1}. For $n\geq 1$ let $V_n:=\{1,\ldots, n \}$. Also for those $n$ for which $dn$ is even, we let $P:=V_n \times [d]$. We call the elements of $P$ \emph{clones}.
A \emph{configuration} is a perfect matching on~$P$.
If we project a configuration onto~$V_n$, then we obtain a $d$-regular multigraph on~$V_n$. Let~$\MG$ denote the multigraph that is obtained by choosing the configuration on~$P$ uniformly at random. It
can be shown (see e.g.~\cite[p.~236]{JLR}) that if we condition on~$\MG$
being simple (i.e.~it does not have loops or multiple edges), then this is distributed
uniformly among all $d$-regular graphs on~$V_n$. In other words, $\MG$ conditional
on being simple has the same distribution as~$\G$. 
Moreover, Corollary 9.7 in~\cite{JLR} guarantees that
\begin{equation}\label{eq:LimSimple}
\lim_{n\rightarrow \infty} \prob ( \MG \mbox{ is simple} ) >0.
\end{equation}
(Of course the above limit is taken over those~$n$ for which~$dn$ is
even.) Let~$A_n$ be a subset of the set of $d$-regular multigraphs on~$V_n$.
Altogether the above facts imply that
if $\prob(\MG \in A_n)\rightarrow 0$ as $n\rightarrow \infty$
then also $ \prob (\G \in A_n)\rightarrow 0$.
This allows us to work with~$\MG$ instead of~$\G$ itself.

\subsection{Some useful facts}
We continue by introducing some notation. Let $G$ be a graph, and let $S,S'$ be subsets of its vertices. Then we denote by $e_G(S)$ the
number of edges in $G$ joining vertices only in $S$, and by $e_G(S, S')$ the number of edges in $G$ joining a vertex in $S$ to a vertex 
in $S'$. Moreover, we denote by $\Gamma_G(v)$ the set of neighbors of a vertex $v$ in $G$.
\begin{lemma}
\label{lem:propMatch}
Let $\A, \B \subseteq V_n \times [d]$ be two disjoint sets of clones, and let $\C \subseteq V_n$ be a set of vertices such that $(\C\times [d]) \cap (\A \cup \B) = \emptyset$. Let $M$ be a matching drawn uniformly at random from the set of perfect matchings on the union of the clones in $\A, \B$ and $\C \times [d]$, and set $N := |\A| + |\B| + d|\C| - 1$. Then
\begin{equation}
\label{eq:exEdgesMatching}
	\ex(e_M(\A)) = \binom{|\A|}{2}\frac1N,
	\quad
	\ex(e_M(\A,\B)) = |\A||\B|\frac1N, \text{ and }
	\ex(e_M(\A,\C)) = d|\A||\C|\frac1N.
\end{equation}
Moreover, let~$H_\ell$ denote the number of vertices in~$\C$ that are adjacent to exactly~$\ell$ clones in~$\A$ in~$M$, where~$0 \le \ell \le d$. Then, if~$|\B| \ge |\A| = \omega(\ln n)$
\begin{equation}
\label{eq:HellM}
	\ex(H_\ell) = \left(1 + o\Big(\frac1{\ln n}\Big)\right) \cdot |\C|\binom{d}{\ell}\left(\frac{|\A|}{N}\right)^\ell\left(1 - \frac{|\A|}{N}\right)^{d-\ell}.
\end{equation}
Finally, let $Q = \sum_{\ell \ge 2}H_\ell$. Then, if $N \ge 4$
\begin{equation}
\label{eq:exTwiceMatching}
	\ex(Q) \le d^2|\A|^2|\C|N^{-2}.
\end{equation}
Let~$X$ be any of~$e_M(\A), e_M(\A, \B), e_M(\A, \C)$ or~$H_\ell$, and let~$\mu = \ex(X)$. Then, if~$\mu = \omega(\ln^2n)$, for any~$\eps = \omega(\mu^{-1/2})$
\begin{equation}
\label{eq:concGeneric}
	\prob(|X - \mu| \ge \eps \mu) \le 4e^{-\frac{\eps^2}{64d(1+\eps)}\mu}.
\end{equation}
\end{lemma}
\begin{proof}
Let~$e, e'$ be edges whose endpoints are in the union of the clones in~$\A, \B$ and~$\C$, and let~$I_e, I_{e'}$ be the indicator variables for
the events that~$e \in M$ and~$e'\in M$. As the number of matchings with~$e$ is equal to the number of matchings with~$e'$ we have
$\ex(I_e) = \ex(I_{e'})$. Hence, as~$\sum_e I_e = \frac{N+1}{2}$ always, we infer that~$\ex(I_e) = \frac1N$. By linearity of expectation this
proves~\eqref{eq:exEdgesMatching}.

To see~\eqref{eq:exTwiceMatching} let $I_{e,e'}$ be the event that both $e$ and $e'$ are in $M$. Note that if $e \cap e' \neq \emptyset$
and also $e \neq e'$, then $\ex(I_{e, e'}) = 0$. Otherwise, let $f, f'$ be any two edges satisfying $f \cap f' = \emptyset$ and $f \neq f'$.
Then, as the number of matchings with $e, e'$ is equal to the number of matchings with $f, f'$ we infer that $\ex(I_{e,e'}) = \ex(I_{f, f'})$. As
$\sum_{e \neq e'} I_{e,e'} = \frac{N+1}2\frac{N-1}2$ always and as there are $3\binom{N+1}{4}$ ways to choose~$e, e'$  such that 
$e\cap e' = \emptyset$ and $e \neq e'$ we obtain that
\[
	\ex(I_{e,e'})
	= \begin{cases}
		0 & , \text{if } e\cap e' \neq \emptyset \text{ and } e \neq e'\\
		\frac1N & , \text{if } e = e'\\
		\frac2{N(N-2)}  & , \text{otherwise}
	\end{cases}.
\]
Let $v = \{v_1, \dots, v_d\}$ be any vertex in $\C$. Moreover, let now $e, e'$ be distinct edges with one endpoint in~$\A$ and the other in 
$v$, and note that there are~$\binom{|\A|}{2} \cdot |\C|\binom{d}{2}$ ways to choose $e$ and~$e'$. If $N\ge 4$, then $\ex(I_{e,e'}) \le
4N^{-2}$, and this completes the proof of~\eqref{eq:exTwiceMatching}.

To see~\eqref{eq:HellM} let~$v\in\C$ and denote by~$L_v$ the event that there is an edge in~$M$ connecting two clones of~$v$. Moreover, let~$H_\ell(v)$ denote the event that~$v$ is adjacent to exactly~$\ell$ clones in~$\A$. Then
\begin{equation}
\label{eq:Hellv}
	\prob(H_\ell(v)) = \prob(H_\ell(v) \cap \overline{L_v}) + \prob(H_\ell(v) ~|~ L_v)\prob(L_v).
\end{equation}
We estimate the above probabilities one by one. We shall begin with $\prob(L_v)$. Note that there are at most $d^2$ choices for an edge that connects two clones of $v$, and that the probability that such an edge is in $M$ is $\frac1N$. Hence,
\begin{equation}
\label{eq:prLv}
	\prob(L_v) \le d^2N^{-1} = o(\ln^{-1}n).
\end{equation}
Next we estimate $\prob(H_\ell(v) \cap \overline{L_v})$. Let us for the moment fix $\ell$ clones $c_1, \dots, c_\ell$ in $\A$, and $\ell$ clones $c'_1, \dots, c'_\ell$ of $v$. Note that there are $\binom{|\A|}{\ell}$ choices for the $c_i$'s and $\binom{d}{\ell}$ choices for the~$c'_i$'s. Then the number of matchings where the $c_i$'s are matched to the $c'_i$'s, and no one of the remaining clones of $v$ is matched to a clone in $\A$, and there is no edge connecting two of the clones of $v$, is $\ell! \cdot \binom{|\B| + d(|\C|-1)}{d-\ell}(d-\ell)! \cdot M_{|\A| + |\B| + d(|\C|-2)}$, where $M_n = \frac{n!}{(n/2)!2^{n/2}}$ denotes the number of perfect matchings on $n$ vertices. Stirling's formula yields the approximation
\begin{equation}
\label{eq:numMatchings}
	M_n = (1 + \Theta(n^{-1})) \cdot \sqrt{2}n^{n/2}e^{-n/2}.
\end{equation}
Moreover, our assumption $|\B| \ge |\A| = \omega(\ln n)$ implies that
\[
	\binom{|\A|}{\ell} = \left(1 + o\Big(\frac1{\ln n}\Big)\right) \cdot \frac{|\A|^\ell}{\ell!}
	\text{ and }
	\binom{|\B| + d(|\C|-1)}{d-\ell} = \left(1 + o\Big(\frac1{\ln n}\Big)\right) \cdot \frac{(N-|\A|)^{d-\ell}}{(d-\ell)!}.
\]
All the above facts together yield that
\begin{equation*}
	\prob(H_\ell(v) \cap \overline{L_v}) = \left(1 + o\Big(\frac1{\ln n}\Big)\right) \cdot \binom{d}{\ell} |\A|^\ell (N-|\A|)^{d-\ell} \cdot \frac{M_{N+1-2d}}{M_{N+1}}.
\end{equation*}
By applying the estimate for $M_n$ we infer that the last fraction equals
\[
	\left(1 + o\Big(\frac1{\ln n}\Big)\right) \cdot e^d \frac{(N+1-2d)^{\frac{N+1-2d}{2}}}{(N+1)^{\frac{N+1}2}}
	= \left(1 + o\Big(\frac1{\ln n}\Big)\right) \cdot N^{-d}.
\]
So,
\begin{equation}
	\prob(H_\ell(v) \cap \overline{L_v}) = \left(1 + o\Big(\frac1{\ln n}\Big)\right) \cdot \binom{d}{\ell} \left(\frac{|\A|}{N}\right)^\ell\left(1 - \frac{|\A|}{N}\right)^{d-\ell}.
\end{equation}
Finally, we estimate $\prob(H_\ell(v) ~|~ L_v)$. Note that the event $H_\ell(v)$, given $L_v$, implies that there are $\ell$ clones of $v$ that are matched to some clones in $\A$. By a similar reasoning as above we infer that
\[
	\prob(H_\ell(v) ~|~ L_v)
	\le \frac{\binom{d}{\ell} \binom{|A|}{\ell} \ell!  \cdot M_{N+1-2\ell}}{M_{N+1}}
	\le \left(1 + o\Big(\frac1{\ln n}\Big)\right)\binom{d}{\ell}\left(\frac{|\A|}{N}\right)^\ell.
\]
Note that our assumption $|\B| \ge |\A| $ implies that $\frac{|\A|}{N} \le \frac{|\A|}{|\A| + |\B|} \le \frac12$. So, $1 - \frac{|\A|}{N} \ge \frac12$, and~\eqref{eq:prLv} together with~\eqref{eq:Hellv} imply that
\[
	\prob(H_\ell(v) ~|~ L_v)\prob(L_v) = o\left(\frac{\prob(H_\ell(v) \cap \overline{L_v})}{\ln n}\right).
\]
By plugging this into~\eqref{eq:Hellv} we thus complete the proof of~\eqref{eq:HellM}.

We finally prove the concentration of $X$ by applying Theorem~\ref{thm:McDiarmid} as follows. We will first specify the families $\Xf$ and
$\pi$. Here, $\Xf = \emptyset$. The random permutation $\pi$ corresponds to the random perfect matching on the union of the vertices in 
$\A, \B$ and $\C$. More precisely, assuming that this union consists of $2k$ clones, which are labeled $1,\ldots, 2k$, 
we consider a uniformly random permutation of these clones $\pi:=(i_1 i_2 \ldots i_{2k-1} i_{2k})$. Then we match the clones that are in
consecutive pairs, that is, we choose the matching $\{(i_1,i_2),(i_3,i_4),\ldots, (i_{2k-1},i_{2k}) \}$. This is a uniform
perfect matching on these clones. Note that the pair $(\Xf, \pi)$ determines the value of $X$. Moreover,
\begin{itemize} 
\item if we swap two elements of $\pi$, then $X$ can change by at most 2; 
\item if $X = \ell$, then we need to specify at most $d\ell$ elements of $\pi$ in order to certify this. 
\end{itemize}
Thus, we may take $c=2$ and $r=d$ in Theorem~\ref{thm:McDiarmid}. Moreover, let $M_X$ be the median of $X$. An easy calculation shows that $|M_X - \ex(X)| = O(\sqrt{\ex(X)})$ (cf. Example 2.33 in~\cite{JLR}). The proof completes by applying Theorem~\ref{thm:McDiarmid} with, say, $t = 1.1\eps M_X$. 
\end{proof}

\section{Analysis of the Randomized Broadcasting Algorithm}

\subsection{The preliminary phase} 

Let $T_0$ be the first round in which the number of informed vertices exceeds $\ln^7n$. We will show the following statement; it is not best possible, but it suffices for our purposes.
\begin{lemma}
\label{lem:prelim}
With probability $1-o(1)$ we have that $T_0 = O(\ln\ln n)$. Moreover, for sufficiently large $n$ the subgraph induced by the vertices in $\I_{T_0}$ is with probability $1- o(1)$ a tree.
\end{lemma}
\begin{proof}
Let $\D_i$ denote the number of vertices at distance $i$ from vertex 1. We will first show that whp.\ we have~$|\D_i| = d(d-1)^{i-1}$ for all~$1 \le i \le \sqrt{\ln n}$, which implies that the subgraph induced by~$\cup_{i=1}^{\sqrt{\ln n}} \D_i$ is whp.\ a tree. To see the claim, we work in the configuration model and expose the sets $\D_i$ one after the other, i.e., we first expose the edges in the random matching that contain the clones of vertex 1, then the edges that contain the (remaining) clones of the vertices in $\D_1$, and so on.

Suppose that~$|\D_i| = d(d-1)^{i-1}$ for all~$i \le j < \sqrt{\ln n}$. This implies that all edges in the matching incident to the clones corresponding to the vertices in $\D_1, \dots \D_{j-1}$ are exposed. Moreover, for every vertex in $\D_j$ there is precisely one clone whose neighbor is exposed, and for all other $d-1$ it is not. Let us denote by $\F_j$ this set of unexposed clones. We have~$|\F_j| = d (d-1)^j$, and let us note for future reference that with room to spare~$|\F_j| \le n^{1/3}$.

Clearly, $\D_{j+1}$ consists of all vertices in $\C = V_n\setminus (\D_1 \cup \dots \cup \D_j)$ for which at least one of their clones is connected in the matching to some clone in $\F_j$. Let $Q$ denote the number of such vertices with the property that they are matched to at least two clones in $\F_j$, and let $M$ be a random perfect matching on the union of the clones in $\F_j$ and $C$. Then
\[
	|\D_{j+1}| = |\F_j| - 2e_M(\F_{j}) - Q.
\]
By applying Lemma~\ref{lem:propMatch} with $\A = \F_j$, $\B = \emptyset$ and $\C$ as above we obtain for large $n$ that
\[
	\ex(e_M(\F_{j})) \le n^{-1/3}
	\quad \text{ and } \quad
	\ex(Q) \le 2d^2 n^{-1/3}.
\]
So, with probability at least $1 - 3d^2n^{-1/3}$ we have that $e_M(\F_{j}) = Q = 0$. The proof of the claim completes by applying the above argument for $i = 1 \dots \sqrt{\ln n}$.

With the above fact we can prove the lemma as follows. Let~$v$ be a vertex in~$\D_i$, for some~$1\le i \le 10\ln\ln n := \ell$, and denote by~$T_v$ the time until~$v$ gets informed. Let~$v'$ be the unique neighbor of~$v$ in~$\D_{i-1}$. Then~$T_v = T_{v'} + X_{v,v'}$, where~$X_{v,v'}$ is a geometrically distributed random variable with success probability~$d^{-1}$. Moreover,~$X_{v,v'}$ is independent of~$T_{v'}$. In other words, we have that~$T_v = \sum_{j=1}^i X_j$, where the~$X_j$'s are iid.\ variables as above. So,~$\ex(T_v) = di$, and by Theorem~\ref{thm:Chernoff}
\[
	\prob(T_v \ge 20d^2i) = \prob(Bin(20d^2i, d^{-1}) < i) \le 2e^{-\frac{(19di)^2}{4 \cdot 20 di}} \le 2e^{-4di}.
\]
In particular, for~$i = \ell$, this probability is at most~$2\ln^{-40d}n$. Moreover, the total number of vertices in~$\cup_{i=1}^{\ell} \D_i$ is at most~$d\frac{d^\ell -1}{d-2} \le \ln^{20d}n$. So, by Markov's inequality, there is no vertex at distance at most~$\ell$ from vertex 1 that will not be informed in the first~$20d^2\ell = O(\ln\ln n)$ rounds. Moreover,~$d\frac{d^\ell -1}{d-2} = \omega(\ln^7n)$, and the proof is completed.
\end{proof}

\subsection{The Exposure Strategy} 

In this section we will describe our general strategy for determining the probable broadcast time of the randomized rumor spreading protocol.
We will denote by ~$\I_t$ the set of informed vertices and by~$\U_t$ the set consisting of the uninformed vertices, i.e.,~$\U_t = [n]\setminus \I_t$, at the beginning of round~$t$. 
We have that~$\I_1 = \{1\}$. We can simulate the execution of the rumor spreading protocol as follows in two steps. First, we choose one of the clones of vertex 1 uniformly at random, say~$c_1$. Then, we expose the edge in the random matching whose one endpoint is~$c_1$, and pass the message to the other endpoint, say~$c_2$. Note that this is equivalent to selecting uniformly at random a clone $c'$ different from $c_1$, and joining $c_1$ and $c'$ by an edge. Clearly,~$c_2$ is a clone that corresponds to some vertex in the original graph, which now becomes informed. This completes the first round, and $\I_2$ consists of vertex 1 and the vertex corresponding to~$c_2$.

This gradual exposure of the graph can be generalized to any other round in the following manner. Suppose that we are in the beginning of round $t+1 \ge 0$. We will simulate the execution of the protocol as follows in two steps.
\begin{enumerate}
\item [{\em Step 1}.]
For each $v \in \I_t$ we choose one of its clones uniformly at random, independently for every such vertex. We shall denote the selected clone by $c_v = c_v(t)$.
\item [{\em Step 2}.] Set $\I_{t+1} = \I_t$ and let $v \in \I_t$. If $c_v$ belongs to an edge in the random matching that was exposed in one of the previous rounds, do nothing. Otherwise, choose uniformly at random one of the remaining unmatched clones, say $c$, and connect it to $c_v$ by an edge. Add the vertex corresponding to $c$ to $\I_{t+1}$, if it isn't already contained in $\I_{t+1}$.
\end{enumerate}
If a clone of a vertex in $\U_t$ is matched to $c_v$, for some $v \in \I_t$, then that vertex becomes \emph{informed} -- we denote by $\N_{t+1}$ the set of those vertices. In short, $\N_{t+1}$ is the set of \emph{newly informed} vertices in the $t+1$st round. 
Let us introduce some further notation regarding the two exposure steps. At the beginning of round $t+1$, we denote by $\P_t$ the set of clones of the vertices in $\I_t$ whose neighbors have not been exposed yet (i.e., in none of the previous rounds the edges in the matching containing those clones were exposed). Among those, during Step~1 we choose a set $\A_{t+1} \subseteq \P_t$ of clones. Informally,
 $\A_{t+1}$ contains the clones through which new vertices might get informed. Finally, we write $N_{t+1} = |\N_{t+1}|$, $A_{t+1} = 
|\A_{t+1}|$ and~$P_t = |\P_t|$, and note that $\P_0$ consists of the $d$ clones of vertex 1.

~\\
The two steps of our exposure strategy can be also viewed as follows. In the first step we choose according to the rule described above a random subset $\A_{t+1}$ of $\P_t$. Then, in Step~2, the clones in $\A_{t+1}$ are matched to the union of the clones in $\P_t$ and the clones corresponding to the vertices in $\U_t$ (as, per definition, all other clones are already matched). In other words, we consider a random perfect matching $\M_{t+1}$ on the set of clones in $\P_t$ and $\U_t$, and we will study its combinatorial properties. In particular, the following claim relates the random quantities in question.
\begin{proposition}
\label{prop:deterministicEvolution}
Let $H_{i, t+1}$ denote the number of vertices in $\U_t$ that were informed $i$ times in round $t+1$, i.e., a vertex $v$ is counted in 
$H_{i, t+1}$, if there are $i$ clones in $\A_{t+1}$ that are matched to the clones of $v$ in $\M_{t+1}$. Then
\begin{eqnarray}
	N_{t+1} &=& \sum_{i = 1}^d H_{i, t+1} \le e_{\M_{t+1}}(\A_{t+1}, \U_t), \label{eq:Nt} \\ 
	I_{t+1} &=& I_t + N_{t+1} \quad\text{ and }\quad U_{t+1} = U_t - N_{t+1}, \label{eq:changeIt} \\
	P_{t+1} &=& P_{t} - A_{t+1} - e_{\M_{t+1}}(\A_{t+1}, \P_{t} \setminus \A_{t+1}) + \sum_{i=1}^d (d-i)H_{i, t+1}. \label{eq:changePt}
\end{eqnarray}
\end{proposition}
\begin{proof}
The first equality in Equation~\eqref{eq:Nt} follows directly from the definition of $H_{i, t+1}$, as every vertex in $\U_t$ has $d$ unmatched clones, and it becomes informed as soon as at least one of them gets matched in Step~2 to a clone in $\A_{t+1}$. The upper bound is also easy to see, as the number of newly informed vertices is at most the number of edges in $\M_{t+1}$ that have one endpoint in $\A_{t+1}$ and the other in the set of clones corresponding to the vertices in $\U_t$.
Equation~\eqref{eq:changeIt} follows immediately from the definition of $\I_t$ and $\N_{t+1}$. Finally, to see~\eqref{eq:changePt}, note first that all clones in $\A_{t+1}$ are excluded from $\P_{t+1}$, as they are matched to other clones in $\P_t$ or $\U_t$; this accounts for the ``$-A_{t+1}$'' term. Moreover, all clones in $\P_t \setminus \A_{t+1}$ that are contained in edges of $\M_{t+1}$ with the other endpoint in $\A_{t+1}$ are excluded from $\P_{t+1}$ as well, as the edge including them was exposed; this accounts for the the ``$- e_{\M_{t+1}}(\A_{t}, \P_{t} \setminus \A_{t+1})$'' term. Finally, for each newly informed vertex counted in $H_{i, t+1}$, i.e., which was informed $i$ times in round $t$, the number of clones counted in $\P_t$ increases by $d-i$.
\end{proof}

For future reference we prove already here a lemma that addresses the concentration properties of $A_{t+1}$.
\begin{lemma} \label{lem:A_tConc} 
For any $t \ge 1$ and $n \ge 5$
\[
\prob\left( \Big|A_{t} - \frac{P_t}d\Big| \ge \frac{P_t}{d\ln^2n} ~\Big|~ P_t\right) \le 2e^{- \frac{P_t}{3d\ln^4 n}}.
\]
\end{lemma}
\begin{proof}
For each clone $c\in\P_t$ let $I_c$ be the indicator variable for the event that $c$ is selected in the first step of the $t$th round, i.e., ``$I_c = 1$'' iff the random decisions in Step~1 are such that $c \in \A_{t+1}$ . Since each clone has probability $1/d$ to be selected we have $\ex(I_c) = 1/d$. Moreover, for two distinct clones $c, c'$ we have that
\[
	\ex(I_c I_{c'})
	=
	\begin{cases}
		0&, \text{ if $c,c'$ are clones belonging to the same $v\in V_n$}\\
		1/d^2&, \text{ otherwise }
	\end{cases}
	\le \frac1{d^2} = \ex(I_c)\ex(I_{c'}),
\]
i.e., the $I_c$'s are negatively correlated. We infer that $\mu := \ex (A_{t+1} \ | \ P_t) = {P_t \over d}$, and Theorem~\ref{thm:Chernoff} implies that the sought probability is at most
\begin{equation*}  
\prob(|A_{t+1} - \mu| \ge \mu/\ln^2n ~|~ P_t)
\le
2 \exp \left( - {\mu^2 \ln^{-4}n \over 2 (\mu + \mu/(3\ln^2n))} \right)
\le 
2 \exp \left( - \frac{\mu}{3\ln^4n}\right).
\end{equation*}
\end{proof}

\subsection{The Middle Phases}

Let $T_1$ be the first round where the number of informed vertices is at least $n - \ln^7n$, or equivalently, where $U_{T_1} \le \ln^7n$. The main accomplishment of this section is the proof of the following lemma, which describes the likely evolution of the number of (un)informed vertices and of $P_t$ until $t = T_1$.
\begin{lemma}
\label{lem:probableEvolution}
Suppose that $P_t, U_t \ge \ln^7n$. Abbreviate $F_t = 1-\frac{P_t}{d(P_t + dU_t)}$.
Then, uniformly with probability at least $1 - o(\frac1{\ln n})$,
\begin{eqnarray}
	P_{t+1} & =& \left(1 - o\Big(\frac1{\ln n}\Big)\right) \cdot \left(\left(1 - \frac1d\right)F_t\cdot P_t + dU_t(F_t -F_t^d)\right), \label{eq:evolPt} \\ 
	U_{t+1} &=& \left(1 - o\Big(\frac1{\ln n}\Big)\right) \cdot  F_t^d \cdot U_t \label{eq:evolUt}.
\end{eqnarray}
\end{lemma}
\begin{proof}
Let $H_{i, t+1}$ denote the number of vertices in $\U_t$ that were informed $i$ times in round $t+1$, and recall that Proposition \ref{prop:deterministicEvolution} describes the relation of the quantities $P_{t+1}$ and $U_{t+1}$ to $P_t, U_t$ and $H_{i, t+1}$. We will show that uniformly for all $t$ such that $P_t,U_t \ge \ln^7 n$, with probability $1 - o(\frac1{\ln^2n})$ we have
\begin{equation}
\label{eq:probableA}
	A_{t+1} = \left(1 + o\Big(\frac1{\ln n}\Big)\right)\frac{P_t}d,
\end{equation}
and	
\begin{equation}
\label{eq:probablee}
	e_{\M_{t+1}}(\A_{t+1}, \P_t\setminus \A_{t+1}) = \left(1 + o\Big(\frac1{\ln n}\Big)\right)\left(1 - \frac1d\right)P_t(1 - F_t) \pm \ln^5 n,
\end{equation}
and that for all $1 \le i \le d$
\begin{equation}
\label{eq:probableH}
	H_{i, t+1} =  \left(1 + o\Big(\frac1{\ln n}\Big)\right)U_t \cdot \binom{d}{i}(1 - F_t)^i F_t^{d-i}\pm \ln^{5}n.
\end{equation}
This proves~\eqref{eq:evolPt} and~\eqref{eq:evolUt} as follows. First, by using~\eqref{eq:Nt} we infer that with probability $1 - o(\frac1{\ln n})$ the number of informed vertices in round $t+1$ is
\[
	N_{t+1} = \sum_{i=1}^d H_{i, t+1} = \left(1 + o\Big(\frac1{\ln n}\Big)\right)U_t \cdot (1 - F_t^d)\pm d\ln^5n.
\]
So, as $U_t \ge \ln^7n$, with probability $1 - o(\frac1{\ln n})$ the number of uninformed vertices at the end of round $t+1$ is
\[
	U_{t+1}
	= U_t - N_{t+1} = U_t - \left(1 + o\Big(\frac1{\ln n}\Big)\right)U_t \cdot (1 - F_t^d)\pm d\ln^5n
	= \left(1 + o\Big(\frac1{\ln n}\Big)\right)F_t^dU_t.
\]
This shows~\eqref{eq:evolUt}. To see~\eqref{eq:evolPt} recall~\eqref{eq:changePt} and note that with probability $1 - o(\frac1{\ln n})$
\[
	\sum_{i=1}^d(d-i)H_{i, t+1} = \left(1 + o\Big(\frac1{\ln n}\Big)\right)U_td(F_t - F_t^d) \pm d\ln^5n.
\]
Hence, by applying~\eqref{eq:changePt} we infer that with probability $1 - o(\frac1{\ln n})$
\[
\begin{split}
	P_{t+1} &= P_t - A_{t+1} - e_{\M_{t+1}}(\A_t, \P_t\setminus \A_t) + \sum_{i=1}^d(d-i)H_{i, t+1} \\
	&= \left(1 + o\Big(\frac1{\ln n}\Big)\right)\left(P_t - \frac{P_t}d - \left(1 - \frac1d\right)P_t(1 - F_t) + U_td(F_t - F_t^d)\right),
\end{split}
\]
and this shows~\eqref{eq:evolPt}.

It remains to prove~\eqref{eq:probableA}-\eqref{eq:probableH}. We start with~\eqref{eq:probableA}. This is easily seen to hold, by applying Lemma~\ref{lem:A_tConc} and using the fact that $P_t\ge\ln^7n$. To see~\eqref{eq:probablee} we apply Lemma~\ref{lem:propMatch} with $\A = \A_{t+1}, \B = \P_t \setminus \A_{t+1}$ and $\C = \U_t$. We infer that
\[
	\mu := \ex(e_{\M_{t+1}}(\A_{t+1}, \P_t\setminus \A_{t+1})) = \frac{A_{t+1}(P_t - A_{t+1})}{P_t + dU_t - 1}.
\]
Note that for sufficiently large $n$ we have with probability $1 - o(\frac1{\ln n})$ that $|\P_t \setminus \A_{t+1}| \ge |\A_{t+1}|$, and that $A_{t+1} = \omega(\ln n)$. By using~\eqref{eq:probableA} and the definition $F_t = 1 - \frac{P_t}{d(P_t + dU_t)}$ we thus obtain 
\[
\mu = \left(1 + o\Big(\frac1{\ln n}\Big)\right)\frac{(1 - \frac1d)P_t^2}{d(P_t + dU_t - 1)}
= \left(1 + o\Big(\frac1{\ln n}\Big)\right)\left(1 - \frac1d\right)P_t(1 - F_t).
\]
If $\mu \ge \ln^3n$, then by applying~\eqref{eq:concGeneric} with $\eps = \ln^{-1.1}n$ we infer that
\[
	\prob(|e_{\M_{t+1}}(\A_{t+1}, \P_t\setminus \A_{t+1}) - \mu| \ge \mu\ln^{-1.1}n) = o(\ln^{-1}n).
\]
On the other hand, if $\mu \le \ln^3n$, we obtain by Markov's inequality that
\[
	\prob(e_{\M_{t+1}}(\A_{t+1}, \P_t\setminus \A_{t+1})\ge \ln^{5}n) = o(\ln^{-1}n).
\]
By combining the above statements we infer that with probability at least $1 - o(\frac1{\ln n})$ we have that $e_{\M_{t+1}}(\A_{t+1}, \P_t\setminus \A_{t+1}) = (1 + o(\frac1{\ln n}))\mu \pm \ln^5n$ i.e.,~\eqref{eq:probablee} is proved.

The proof of \eqref{eq:probableH} is very similar. By applying Lemma~\ref{lem:propMatch} with $\A = \A_{t+1}, \B = \P_t \setminus \A_{t+1}$ and $\C = \U_t$ we infer that
\[
	\mu_i := \ex(H_{i, t+1}) = \left(1 + o\Big(\frac1{\ln n}\Big)\right) U_t \cdot \binom{d}{i}\left(\frac{A_t}{P_t + dU_t-1}\right)^i\left(1 - \frac{A_t}{P_t + dU_t-1}\right)^{d-i}.
\]
As with probability $1 - o(\frac1{\ln n})$ we have $A_{t+1} = (1 + o(\frac1{\ln n}))\frac{P_t}d$ we infer that $$\mu_i = \left(1 + o\Big(\frac1{\ln n}\Big)\right) U_t \cdot \binom{d}{i} F_t^i (1 - F_t)^{d-i}.$$ The proof now completes with a case distinction as above, i.e., we treat the case $\mu_i \le \ln^5n$ with Markov's inequality and the case $\mu_i \ge \ln^5n$  by using~\eqref{eq:concGeneric}.
\end{proof}
Lemma~\ref{lem:probableEvolution} allows us now to derive probable bounds for $T_1$.
\begin{corollary}
\label{cor:middle}
With probability $1-o(1)$ we have that $T_1 - T_0 = C_d \ln n + O(\ln\ln n)$, where
\[
	C_d = \frac1{\ln(2(1-\frac1d))} - \frac1{d\ln(1 - \frac1d)}.
\]
\end{corollary}
\begin{proof}
By applying Lemma~\ref{lem:prelim} we infer that at round $T_0$ with high probability there are for the first time at least $\ln^7n$ informed vertices, and the set of informed vertices induces a tree. Hence, we may assume that
\[
	\ln^7 n\le I_{T_0} \le 2\ln^7n
	\quad\text{and}\quad
	(d-1)I_{T_0}\le P_{T_0} \le dI_{T_0}.
\]
We will use those facts in the sequel without further reference.

Let $p_t$ and $u_t$ be given by the recursions
\[
	p_{t+1} = \left(1 - \frac1d\right)f_tp_t + du_t(f_t - f_t^d)
	\quad	\text{ and } \quad
	u_{t+1} = f_t^d u_t,
\]
where $f_t = 1 - \frac{p_t}{d(p_t + du_t)}$, and $p_{T_0} = P_{T_0}, u_{T_0} = n - I_{T_0}$. As we are interested in the probable values of $P_t$ and $U_t$ for $t = O(\ln n)$ we infer by applying Lemma~\ref{lem:probableEvolution} that $p_t = (1 + o(1))P_t$ and $u_t = (1 + o(1))U_t$ for all such $t$, provided that $U_t, P_t \ge \ln^7n$. In the sequel we shall therefore consider only the evolution of $p_t$ and $u_t$.

Let $q := 2\left(1 - \frac1d\right)$, $\eps = 0.01$ and $t_1$ be the minimal $t$ such that $q^{t - T_0} \le \frac{\eps n}{\ln^7n}$. We will first show that for all $T_0 \le t\le t_1$
\begin{equation}
\label{eq:pfirstphase}
	p_t \le P_{T_0} \cdot q^{t-T_0}
	\quad \text{ and } \quad
	p_t \ge P_{T_0} \cdot q^{t-T_0}	- 3P_{T_0}^2 \cdot q^{2(t-T_0)}/n,
\end{equation}
and
\begin{equation}
\label{eq:ufirstphase}
	u_t = n - I_{T_0} - P_{T_0}\frac{q^{t-T_0}-1}{d(q-1)} \pm 9 \cdot P_{T_0}^2q^{2(t-T_0)}/n.
\end{equation}
We proceed by induction on $t$. Note that for $t = T_0$ the statement trivially holds. In order to perform the induction step ($t \to t+1$) we will need some facts. First, let $x = 1-f_t$ and note that
\[
	f_t - f_t^d = (1-x) - (1-x)^d \le (d-1)x = (d-1)\frac{p_t}{d(p_t + du_t)} \le \frac{d-1}{d^2}\frac{p_t}{u_t}.
\]
So, we readily obtain the upper bound for $p_t$ in~\eqref{eq:pfirstphase} by using the the recursion for $p_t$ as follows.
\[
	p_{t+1} \le \left(1 - \frac1d\right)f_tp_t + du_t \cdot \frac{d-1}{d^2}\frac{p_t}{u_t} \le 2\left(1 - \frac1d\right)p_t = qp_t \Rightarrow p_{t+1} \le P_{T_0} \cdot q^{t +1- T_0}.
\]
To see the lower bound for $p_t$, note first that the induction hypothesis, together with the fact that $q^{t-T_0} \le \frac{\eps n}{\ln^7n}$ imply that $\frac{p_t}{u_t} < 1$. Thus, $\frac{1}{1 + \frac{p_t}{du_t}} \ge 1 - \frac{p_t}{du_t}$. A similar calculation as above and by using the fact $(1-x)^d \le 1 - dx + \binom{d}{2}x^2$ for $x \ge 0$ reveals that
\[
\begin{split}
	f_t - f_t^d
	\ge (d-1)x - \binom{d}{2}x^2
	&\ge \frac{d-1}{d}\frac{p_t}{p_t + du_t} - \frac{d^2}{2}\frac{p_t^2}{d^2(p_t + du_t)^2} \\
	&\ge \frac{d-1}{d^2}\frac{p_t}{u_t(1 + \frac{p_t}{du_t})} - \frac{p_t^2}{2d^2u_t^2}
	\ge \frac{d-1}{d^2}\frac{p_t}{u_t} - \frac{3p_t^2}{2d^2u_t^2}.
\end{split}
\]
By using again the recursion for $p_t$ we infer that
\[
	p_{t+1} \ge \left(1 - \frac1d\right)f_tp_t + du_t\cdot\left(\frac{d-1}{d^2}\frac{p_t}{u_t} - \frac{3p_t^2}{2d^2u_t^2}\right) \ge qp_t - \frac{2}{d}\frac{p_t^2}{u_t}.
\]
Note that the induction hypothesis and the fact $q^{t-T_0} \le \frac{\eps n}{\ln^7n}$ imply that $u_t \ge n/2$. So,
\[
\begin{split}
	p_{t+1}
	&\ge qp_t - \frac{4}{dn}p_t^2 \ge P_{T_0}q^{t +1 - T_0} - \frac{3P_{T_0}^2q^{2(t-T_0) + 1}}{n} - \frac{4}{dn}\left(P_{T_0}q^{t-T_0}\right)^2 \\
	&= P_{T_0}q^{t +1 - T_0} - \frac{P_{T_0}^2q^{2(t-T_0+1)}}{n}\left(\frac3q + \frac4{dq^2}\right)
	\ge P_{T_0}q^{t +1 - T_0} - 3\frac{P_{T_0}^2q^{2(t-T_0+1)}}{n}.
\end{split}
\]
This proves the lower bound for $p_t$ in~\eqref{eq:pfirstphase}. Next we prove the bounds for $u_{t+1}$. Note that
\begin{equation*}
\label{eq:ratioUlower}
	\frac{u_{t+1}}{u_t}
	= \left(1 - \frac{p_t}{d(p_t + du_t)}\right)^d
	\ge 1 - \frac{p_t}{p_t + du_t}
	\ge 1 - \frac{p_t}{du_t}
	\quad\Rightarrow\quad
	u_{t+1} \ge u_t - \frac{p_t}{d}.
\end{equation*}
A similar calculation using the fact $(1-x)^d \le 1 - dx + \binom{d}{2}x^2$ for $x \ge 0$ reveals that
\begin{equation*}
\label{eq:ratioUupper}
	\frac{u_{t+1}}{u_t} \le 1 - \frac{p_t}{p_t + du_t} + \binom{d}{2}\frac{p_t^2}{d^2(p_t + du_t)^2}
	\le 1 - \frac{p_t}{du_t} + \frac{3}{4}\frac{p_t^2}{u_t^2}.
\end{equation*}
Recall that the induction hypothesis guarantees  $u_t \ge n/2$. The above facts together with the bounds for $p_t$ imply after a straightforward but lengthy calculation~\eqref{eq:ufirstphase}. We omit the details.

The above discussion settles the growth of $p_t$ and $u_t$ up to the time $t_1$. Note that $t_1 = \ln(2(1-\frac1d))^{-1}\ln n + \Theta(\ln\ln n)$. In order to deal with $t > t_1$ let us first make two important observations. First, note that at $t_1$ we have that
\begin{equation}
\label{eq:ratiot1}
	\frac{p_{t_1}}{u_{t_1}} = \Omega(1).
\end{equation}
Let us next consider the ratio $r_t := p_t/u_t$. Note that $f_t = 1 - \frac{p_t}{d(p_t + du_t)} = 1 - \frac{1}{d(1 + {d}/{r_t})}$. The recursions for $p_t$ and $u_t$ imply that
\[
	r_{t+1} = \left(1 - \frac{1}{d}\right) f_t^{-d+1}r_t + d(f_t^{-d+1}  - 1)
	\Rightarrow
	\frac{r_{t+1}}{r_t} = \left(1 - \frac{1}{d}\right) f_t^{-d+1} + \frac{d}{r_t}(f_t^{-d+1}  - 1).
\]
Consider the function
\[
	g(x) = \left(1 - \frac{1}{d} + \frac{d}{x}\right) \left(1 - \frac{1}{d(1 + {d}/x)}\right)^{-d+1} - \frac{d}{x},
\]
and note that~$\frac{r_{t+1}}{r_t} = g(r_t)$. A straightforward calculation shows that~$\lim_{x\to 0}g(x) = 2(1 - \frac1d)$. In the sequel we will argue that~$g$ is monotone increasing. This implies~$\frac{r_{t+1}}{r_t} \ge g(0) \ge \frac43$, and so we have for any~$t'>0$
\begin{equation}
\label{eq:dramaticIncreaseOfp}
	r_{t + t'} \ge r_t \left(\frac43\right)^{t'} \Rightarrow p_{t + t'} \ge \left(\frac43\right)^{t'} u_{t + t'}.
\end{equation}
This fact will become very useful later on. To see why~$g$ is increasing, note that
\[
	g'(x) = \frac{-T(d^2 + x) + dx + d^2}{x^2 + xd},
	\text{ where }
	T = \left(1 - \frac{1}{d(1 + d/x)}\right)^{-d+1}.
\]
Suppose that there is an~$x_0\ge 0$ such that~$g'(x_0) = 0$. Then~$-T + d^2 = x_0(T-d)$. However, we always have~$1 \le T < d$. Thus, the right-hand side of the above equation is~$<0$, while the left-hand side is~$>0$. We infer that there is no such~$x_0$, and therefore the sign of~$g'(x)$ equals the sign of~$g'(0)$. As the latter is easily seen to be positive, this concludes the proof of the monotonicity of~$g$.

Let~$t_2$ be the minimal~$t$ such that~$p_{t_2} \ge u_{t_2} \ln^2n$. The Equations~\eqref{eq:ratiot1} and~\eqref{eq:dramaticIncreaseOfp} guarantee that~$t_2 = t_1 + O(\ln\ln n)$, and moreover that for any~$t > t_2$ such that~$u_t > 0$ we have~$p_t \ge u_t\ln^2n \ge 1$. Under these conditions note that
\[
	f_t^d = \left(1 - \frac{p_t}{d(p_t + du_t)}\right)^d = \left(1 + O(\ln^{-2}n)\right)\left(1 - \frac1d\right)^d.
\]
Thus, for any~$t$ such that~$t = t_2 + O(\ln n)$ we have that
\[
	u_t = (1 + o(1)) \cdot \left(1 - \frac1d\right)^{d(t - t_2)}u_{t_2}.
\]
Recall that~$T_1$ is the first~$t$ such that~$U_{T_1} \le \ln^7n$. As~$u_{t_2}\le n$, we readily obtain that~$T_1 \le t_1 + O(\ln\ln n) - \frac1{d\ln((1 - \frac1d))}\ln n = C_d\ln n + O(\ln\ln n)$. To see the corresponding lower bound for~$T_1$, note that as long as~$p_t \ge 1$ we always have
\[
	u_{t+1} \ge \left(1 - \frac1d\right)^d u_{t}.
\]
The proof completes with the fact~$u_{t_1} = \Theta(n)$.
\end{proof}
\subsection{The Final Phase}

Let~$T_1$ be the first time such that the number of uninformed vertices drops below~$\ln^7 n$. In the previous section we argued that~$T_1 = C_d\ln n + O(\ln\ln n)$, where~$C_d$ is given in Corollary~\ref{cor:middle}. The main aim of this section is to prove that the broadcasting of the message completes after additional~$O((\ln\ln n)^2)$ rounds. This is shown in the next lemma.
\begin{lemma}
With probability $1-o(1)$ we have $T - T_1 = O((\ln\ln n)^2)$.
\end{lemma}
\begin{proof}
Before we show the claim let us prove a auxiliary fact. Let $S$ be any subset of the vertices of $\MG$ of size at most $\ln^7n$. We will show that with probability $1 - o(1)$
\[
	e(S) < 1.1|S|.
\]
To see the claim, suppose that there is an $S$ such that $e(S) \ge 1.1s$, where we set $s = |S|$. 
There are $\binom{n}{s} \le (\frac{en}{s})^s$ choices for the set $S$. 
Moreover, there are at most $s^{2.2s}$ ways to choose $1.1s$ edges in $S$. Finally, the probability that the chosen edges are in $\MG$ is $\frac{M_{dn - 2.2s}}{M_{dn}}$, where $M_x$ denotes the number of perfect matchings on $x$ vertices. Using~\eqref{eq:numMatchings} we infer that
\[
\begin{split}
	\prob(\exists S: e(S) \ge 1.1|S|)
	&\le (1 + o(1))\left(\frac{en}{s}\right)^s \cdot s^{2.2s} \cdot \frac{e^{dn/2}}{(dn)^{dn/2}} \frac{(dn - 2.2s)^{dn/2 - 1.1s}}{e^{dn/2 - 1.1s}}\\
	&\le (1 + o(1)) (e^{2.1}s^{1.2}n)^s \cdot (dn)^{-1.1s}.
\end{split}
\]
This expression is $n^{-\Omega(1)}$ for any $s \le \ln^7n$; this concludes the proof of the auxiliary claim. In particular, $\MG$ is such that any set $S$ of at most $\ln^7n$ vertices satisfies with room to spare
\[
	e(S, V_n\setminus S) \ge (d-2.2)s \ge ds/4.
\]
With this fact at hand it is routine to complete the proof of the lemma. Indeed, let $S$ be the set of uninformed vertices at some point in time after $T_1$. So, $|S| \le \ln^7n$. As $e(S, V_n\setminus S) \ge ds/4$, we know that at least $s/4$ vertices in $S$ have at least one neighbor in $V_n\setminus S$. More precisely, there is a set $S'\subseteq S$ such that $|S'| \ge s/4$ and for all $v \in S'$ there is at least one $v'\in V_n\setminus S$ such that $v$ and $v'$ are joined by an edge.

Denote by $B$ the event that after $10\ln\ln n$ rounds there is a $v\in S'$ that was not informed by $v'$. The probability for this event is at most
\[
	|S'| \cdot \left(1 - \frac1d\right)^{10 \ln\ln n} \stackrel{(|S'| \le \ln^7n)}{=} o(\ln^{-1}n).
\]
So, after $10\ln\ln n$ rounds the new set of uninformed vertices has size at most $|S\setminus S'| \le \frac34|S|$. Iterating the above argument $O(\ln\ln n)$ times finally completes the proof.
\end{proof}

\section{Randomized broadcasting on expanding graphs: proof of Theorem~\ref{thm:pseudo}}
In this section we prove Theorem~\ref{thm:pseudo}, thus bounding the broadcast time on connected 
$(n,d,\lam)$ graphs with $\lam = O(\sqrt{d})$. In order to avoid any confusion we stress that this condition is interpreted as follows: there is a $C>0$ such that for any $n$ sufficiently large $\lam \leq C\sqrt{d}$. 

We will use the main result from~\cite{FHP09}. Before we state it, let us first introduce the notion of a $(p,\eps)$-\emph{typical} graph. A graph $G=G(V,E)$ on $n$ vertices is called  $(p,\eps)$-\emph{typical}, if the following three conditions are satisfied:
\begin{itemize}
\item For any $S \subseteq V$ with $|S|\geq \eps^2 n$, there is a set $X_S \subseteq V \setminus S$ with $|X_S| \leq {8 n \over \ln n}$
such that 
$$ \forall v \in (V \setminus S)\setminus X_S : d_S(v) = (1\pm \eps ) p |S|.$$
\item For any $S \subseteq V$ with $|S| \leq \eps^2 n$, there is a set $X_S \subseteq V \setminus S$ with $|X_S| \leq \eps |S|$
such that 
$$ \forall v \in (V \setminus S)\setminus X_S : d_S(v) \leq \eps p n.$$
\item For all $S \subseteq V$ we have 
$$ e(S,V\setminus S) = |S|(n-|S|)p\left(1 \pm 8 \sqrt{\eps} \right).$$
\end{itemize}
The following appears in~\cite{FHP09}. 
\begin{lemma}
Let $\eps=\eps(n)$ be a positive real-valued function such that $\eps(n) \rightarrow 0$, as $n \rightarrow \infty$, but 
$\eps \geq \ln^{-1/9} n$. 
Let $p \geq {1\over \eps^2}{\ln n \over n}$. If $G$ is a $(p,\eps)$-typical graph and $v \in V$, then with probability 
$1-o(1)$
$$|T(G,v) - (\log_2 n + \ln n)|\leq 3 \eps^{1/3} \ln n. $$
\end{lemma}

We will show that an $(n,d,\lam)$ graph is $(p,\eps)$-typical with $p=d/n$ and $\eps \geq \ln^{1/9} n$. In particular, we will prove the first two conditions by sampling uniformly at 
random a vertex in $V$, and then showing with Chebyschev's inequality that its degree in a given set $S$ is concentrated around its expected value which, as we shall see, equals $d|S|/n$. 

Let $A$ be the adjacency matrix of $G$ and let $e_1, \ldots, e_n$ be an orthonormal basis of $\mathbb{R}^n$ consisting 
of the eigenvectors of $A$, ordered according to the moduli of the corresponding eigenvalues $\lam_1,\ldots, \lam_n$. 
Since $G$ is $d$-regular  and connected, we have 
$e_1:={1\over \sqrt{n}} [1,\ldots,1]^{T}$ (cf. Proposition 2.10 in~\cite{KrivSud}) and the corresponding eigenvalue is $d$.  
For the sake of notational convenience, we will fix an ordering on $V$, namely $v_1,\ldots, v_n$ and we will assume that the 
$i$th entry of each vector corresponds to $v_i$.

Let $S$ be an arbitrary subset of $V$ and let $\chi_S$ be the characteristic vector of $S$, that is, the vector indexed by $V$ where 
the elements corresponding to the vertices of $S$ are equal to 1 and the remaining ones are equal to 0. 
We set $\de:=A\chi_S$ and note that $\de = [d_S(v_1),\ldots, d_S(v_n)]^{T}$. 

Let $v$ be a vertex in $V$ chosen uniformly at random. Thus $\ex (d_S(v))= {1\over n} \sum_{u \in V} d_S(u)$. 
Note that this sum is just ${\langle\de, e_1\rangle\over \sqrt{n}}$, where $\langle\cdot, \cdot \rangle$ denotes the 
usual dot product in $\mathbb{R}^n$. On the other hand, we can express $\de = A \chi_S$ also by taking the expansion of $\chi_S$ 
with respect to the basis $e_1,\ldots, e_n$ and then multiplying by $A$. Note that $\langle\chi_S, e_1\rangle e_1 = 
{|S|\over n}[1,\ldots,1]^T $.  Thus
$$ \chi_S = {|S| \over n}[1,\ldots,1]^T + \sum_{i\geq 2}\langle\chi_S,e_i\rangle e_i.$$
Therefore 
$$ A\chi_S = {d|S| \over n}[1,\ldots,1]^T + \sum_{i\geq 2}\lam_i \langle\chi_S,e_i \rangle e_i. $$
Since $e_1$ is orthogonal to the vectors $e_2,\ldots, e_n$, we have 
$$ \langle\de, e_1\rangle = \langle A\chi_S, e_1\rangle = {d |S| \over n} \sqrt{n},$$
implying that 
$$ \bar{d}:=\ex (d_S(v)) = {d |S| \over n}.$$ 
In the following we will bound the variance of $d_S(v)$. Write $\mathrm{Var} (d_S(v)) = D/n$, where $D:= \sum_{u \in V}d_S^2 (u)- n\bar{d}^2$. But $\sum_{u \in V}d_S^2(u) = \| \de \|^2$. By Pythagoras' Theorem
\[
	\| \de \|^2 = \sum_{i=1}^n \langle\de,e_i\rangle^2 = n\bar{d}^2 + \sum_{i=2}^n \langle\de,e_i\rangle^2.
\]
Therefore, 
$D=\sum_{i=2}^n \langle\de,e_i\rangle^2$. To bound the latter sum note that
\begin{equation*}
\begin{split}
&\sum_{i=2}^n \langle\de,e_i\rangle ^2 = \sum_{i=2}^n \langle A\chi_S ,e_i\rangle ^2 = \sum_{i=2}^n \langle\chi_S , A e_i\rangle ^2 
= \sum_{i=2}^n \lam_i^2 \langle\chi_S , e_i\rangle ^2 \\ 
& \leq \lam^2 \sum_{i=2}^n \langle\chi_S , e_i\rangle ^2 = \lam^2 \left(\| \chi_S \|^2 - \langle\chi_S,e_1\rangle ^2 \right)
= \lam^2 \left(|S| - {|S|^2 \over n} \right) =\lam^2 |S| \left( 1 - {|S| \over n} \right).
\end{split}
\end{equation*}
Thus 
$$ \mathrm{Var}(d_S(v)) = {D \over n} \leq \lam^2 {|S| \over n}\left( 1 - {|S| \over n} \right). $$
Now we are ready to derive the first two conditions of the definition of $(d/n,\eps)$-typicality, when $\lam \leq C \sqrt{d}$. 
\begin{itemize}
\item Let $S$ be such that $|S|\geq \eps^2 n$. Then the size of $X_S$ is bounded from above by 
$n \prob (|d_S(v) - \bar{d}|> \eps \bar{d})$. We bound this probability with Chebyschev's inequality. Indeed, 
\begin{equation*} 
\begin{split}
\prob (|d_S(v) - \bar{d}|> \eps \bar{d}) \leq {\mathrm{Var}(d_S(v)) \over \eps^2 \bar{d}^2} \leq {\lam^2 |S|/n \over \eps^2 \bar{d}^2}
= {\lam^2 \over d \eps^2 \bar{d}}  \leq {n C^2 \over \eps^2 d |S|} \leq {C^2 \over \eps^4 d}.  
\end{split}
\end{equation*}
By the choice of $\eps$, the above bound is at most $8/\ln n$ and therefore $|X_S | \leq 8 n/\ln n$.
\item   Now let $|S| \leq \eps^2 n$. Thus $\bar{d}\leq d\eps^2$. Here the size of $X_S$ is bounded from above by 
$n\prob (d_S(v) > d\eps )$. Since for $n$ large enough $d\eps  - d \eps^2 > d\eps/2$, this probability is at most 
$\prob (d_S(v) - \bar{d} > d\eps /2)$. Again, Chebyschev's inequality implies
\begin{equation*}
\begin{split}
\prob (d_S(v) - \bar{d} > d\eps /2) \leq {4 \mathrm{Var}(d_S(v)) \over \eps^2 d^2} \leq {4 \lam^2 |S|/n \over \eps^2 d^2} \leq 
{4 \lam^2 \eps^2 \over \eps^2 d^2} = {4 \lam^2 \over d^2} \leq {4 C^2 \over d}. 
\end{split}
\end{equation*}
Thus $|X_S| \leq {4C^2 n \over d}$. We want to deduce that this is at most $\eps |S|$. 
We may assume that $|S|\geq d\eps$, as otherwise what we are aiming at holds trivially. So, it suffices to 
deduce that ${4C^2 n \over d}\leq d \eps$. But this holds by our assumption
$d \geq \sqrt{4C^2 n \ln^{1/9} n} \geq \sqrt{4C^2 \eps^{-1} n }$.  
\end{itemize} 
The third condition in the definition of $(d/n,\eps)$-typicality is a standard property of $(n,d,\lam)$ graphs.
\begin{theorem}[Theorem 2.11 in~\cite{KrivSud}]
Let $G=G(V,E)$ be an $(n,d,\lam)$ graph. Then for any two subsets $U,W \subset V$ we have 
$$ \left| e(U,W) - {d |U| |W| \over n} \right| \leq \lam \sqrt{|U||W| \left( 1- {|U| \over n} \right) \left( 1 - {|W|\over n}\right)}.$$
\end{theorem}
We set $U=S$ and $W=V \setminus S$. Then the above implies
\begin{equation*}
\begin{split}
\left| e (S,V\setminus S) - {d |S|(n-|S|) \over n}\right| & \leq \lam \sqrt{|S|(n-|S|) \left( 1- {|S| \over n} \right) 
\left( 1 - {n-|S|\over n}\right)} \\
& = {\lam |S|(n-|S|) \over n} = {\lam \over d}~{d |S|(n-|S|) \over n} \leq {C \over \sqrt{d}}~{d |S|(n-|S|) \over n}.
\end{split}
\end{equation*}
Since $d\geq \sqrt{{4 C^2 \eps^{-1} n }}$, we have ${C\over \sqrt{d}} \leq {\eps^{1/4} \over 2 \sqrt{n}}$. But 
$\eps \geq \ln^{-1/9} n$ and, therefore, the latter bound is at most $8 \eps^{1/2}$, as required to satisfy the third condition.

\medskip

\noindent
{\bf Acknowledgment} 
We would like to thank Colin McDiarmid for suggesting the use of his concentration inequality (Theorem~\ref{thm:McDiarmid}), 
which greatly facilitated our proofs.

\end{document}